\documentclass[12pt]{amsart}

\usepackage{amsfonts,amsmath,amsthm,amscd,amssymb,latexsym}
\usepackage[english]{babel}

\theoremstyle{plain}

\newtheorem{thm}{Theorem}
\newtheorem{lem}[thm]{Lemma}
\newtheorem{prop}[thm]{Proposition}
\newtheorem{cor}[thm]{Corollary}
\newtheorem{quest}{Question}
\newtheorem{remark}{Remark}

\newtheorem{thmm}{Theorem}

\numberwithin{equation}{section}

\begin{document}

\title{Limits of sequences of continuous functions depending on finitely many coordinates}
\author{Olena Karlova}
\author{Volodymyr Mykhaylyuk}

\address{Department of Mathematical Analysis, Faculty of Mathematics and Informatics, Yurii Fedkovych Chernivtsi National University, 58000 Kotsyubyns'koho str., 2, Chernivtsi, Ukraine}

\begin{abstract}
 We answer two questions from {\it V.Bykov, On Baire class one functions on a product space, Topol. Appl. {199} (2016) 55--62,} and prove that
every Baire one function on a subspace of a countable perfectly normal product is the pointwise limit of a sequence of continuous functions, each depending on finitely many coordinates. It is proved also that a lower semicontinuous function on a subspace of a countable perfectly normal product is the pointwise limit of an increasing sequence of continuous functions, each depending on finitely many coordinates, if and only if the function has a minorant which depends on finitely many coordinates.
\end{abstract}



\maketitle

\section{Introduction}
The collection of all continuous maps between topological spaces $X$ and $Y$ is denoted by ${\rm C}(X,Y)$.
If $A\subseteq Y^X$, then we denote the family of all  pointwise  limits of sequences of maps from $A$ by $\overline{A}^{\,{\rm p}}$.
We say that a subset  $A$ of a topological space $X$ is {\it functionally $F_\sigma$ ($G_\delta$)}, if it is a union (an intersection) of a sequence of functionally closed (functionally open) sets in $X$.

Recall that a map $f:X\to Y$ belongs to the {\it first Lebesgue class}, if the preimage $f^{-1}(V)$ of each open set $V\subseteq Y$ is an $F_\sigma$-set in $X$, and belongs to the {\it first Baire class}, if $f\in \overline{{\rm C}(X,Y)}^{\,{\rm p}}$. We denote the collection of all Lebesgue one maps and Baire one maps between $X$ and $Y$ by ${\rm H}_1(X,Y)$ and ${\rm B}_1(X,Y)$, respectively. If a Lebesgue (Baire) one map $f:X\to Y$ takes finitely many values, then we say that $f\in {\rm H}_1^0(X,Y)$ (${\rm B}_1^0(X,Y)$). Let us observe that ${\rm H}_1(X,\mathbb R)={\rm B}_1(X,\mathbb R)$ for a perfectly normal space $X$ (see \cite[\S 31]{Kuratowski:Top:1}).

Properties of functions defined on products of topological spaces are tightly connected with the dependence of such functions on some sets of coordinates. The dependence of continuous mappings on some sets of coordinates was investigated by many mathematicians (see, for instance \cite{Mazur}, \cite{Engelking}, \cite{Noble-Ulmer} and the literature given there). However,  the dependence on countably many coordinates serves as a convenient tool for the investigation of Namioka property of separately continuous maps and their analogs~\cite{Mykhaylyuk}.  Functions depending locally on finitely many coordinates plays an important role in the theory of smoothness and renorming on Banach spaces~\cite{Hajek}.

Let  $(X_n)_{n=1}^\infty$ be a sequence of topological spaces, $P=\prod_{n=1}^\infty X_n$ and  $a=(a_n)_{n\in\mathbb N}\in P$ be a point. For every $n\in\mathbb N$ and $x\in P$ we put  $p_n(x)=(x_1,\dots,x_n,a_{n+1},a_{n+2},\dots)$.
We say that a set $A\subseteq P$  {\it depends on finitely many coordinates}, if there exists $n\in\mathbb N$ such that for all  $x\in A$ and $y\in P$ the equality $p_n(x)=p_n(y)$ implies $y\in A$. A map $f:X\to Y$ defined on a subspace $X\subseteq P$ is {\it finitely determined} if it {\it depends on finitely many coordinates}, i.e., $f(x)=f(y)$ for all  $x,y\in X$ with  $p_n(x)=p_n(y)$. We denote by ${\rm CF}(X,Y)$ the set of all continuous finitely determined maps between $X$ and $Y$; we write ${\rm CF}(X)$ for ${\rm CF}(X,\mathbb R)$.

Vladimir Bykov in his recent paper \cite{Bykov} proved the following results.
\begin{thmm}\label{thm:Bykov}
Let $X$ be a subspace of a product $P=\prod_{n=1}^\infty X_n$ of a sequence of metric spaces $X_n$.  Then
\begin{enumerate}
  \item\label{stat:1} every Baire class one function $f:X\to\mathbb R$ is the pointwise limit of a sequence of functions in ${\rm CF}(X)$;

  \item\label{stat:2} a lower semicontinuous function $f:X\to\mathbb R$ is the pointwise limit of an increasing sequence of functions from ${\rm CF}(X)$ if and only if $f$ has a minorant in ${\rm CF}(X)$.
\end{enumerate}
\end{thmm}
The following questions were formulated in \cite{Bykov}:
\begin{quest}\label{question:Bykov1}{\rm \cite[Remark 1]{Bykov}}  Does the conclusion~(\ref{stat:1}) of Theorem~\ref{thm:Bykov} holds for   an arbitrary completely regular space $X$?
\end{quest}

\begin{quest}\label{question:Bykov2}{\rm \cite[Remark 2]{Bykov}}  Does the conclusion~(\ref{stat:2}) of Theorem~\ref{thm:Bykov} holds for an arbitrary perfectly normal space $X$?
\end{quest}

We show that that the answer on Question~\ref{question:Bykov1} is negative (see Theorem~\ref{th:example}). We prove that every Lebesgue one function $f:X\to Y$ is a pointwise limit of a sequence of functions from ${\rm CF}(X,Y)$ if $Y$ is a path-connected separable metric R-space (in particular, a convex subset of a normed separable space) and one of the following properties holds:  $X$ is a Lindel\"{o}f subspace of a completely regular product $P$, or $X$ is a subspace of  perfectly normal $P$ (see Theorem~\ref{thm:L_one_is_CF_limit}), or $X=P$ and $P$ is pseudocompact (see Theorem~\ref{thm:pseudocomp}). Theorem~\ref{thm:semicont} gives the positive answer on Question~\ref{question:Bykov2}.

\section{Approximation of functions with finitely many values}\label{sec:approx_simple}

Let $X$ be a topological space. For a countable family $\mathcal F=(f_n:n\in\mathbb N)$ of continuous functions  $f_n:X\to [0,1]$ we put
\begin{gather}\label{gath:pseudoF}
d_{\mathcal F}(x,y)=\sum_{n=1}^\infty \frac{1}{2^n}|f_n(x)-f_n(y)|
\end{gather}
for all  $x,y\in X$. Then $d_{\mathcal F}$ is a pseudometric on $X$ such that every continuous function $f:(X,d_{\mathcal F})\to Y$ is continuous on $X$ in the initial topology for any topological space $Y$.

\begin{lem}\label{cor:char_closed_stableCF} Let $P=\prod_{n=1}^\infty X_n$ be a product of completely regular spaces, $X$ be a subspace of $P$ and $G$ be a functionally open set in $X$. If one of the following conditions holds:
\begin{enumerate}
    \item\label{perfnorm} $P$ is perfectly normal, or

    \item\label{lind} $X$ is Lindel\"{o}f,
  \end{enumerate}
  then
  \begin{enumerate}
    \item[(a)] there exists an increasing sequence of functions  $f_n\in{\rm CF}(P,[0,1])$ such that the characteristic function \mbox{$\chi_G:X\to [0,1]$} is a pointwise limit of the sequence $(f_n|_X)_{n=1}^\infty$;

    \item[(b)] there exists a countable family $\mathcal F$ of functions from ${\rm CF}(P,[0,1])$ such that $G$ is open in $(X,d_{\mathcal F})$.
  \end{enumerate}
\end{lem}

\begin{proof} {\bf (a)}. Let $\tilde G$ be an open set in $P$ such that $G=X\cap\tilde G$. Take a family $(W_{s}:s\in S)$ of basic functionally open sets in $P$ such that
$$
\tilde G=\bigcup_{s\in S}W_s.
$$

Notice that  $G$ is an $F_\sigma$-set in $X$. Consequently,  $G$ is Lindel\"{o}f in case (\ref{lind}) and we choose at most countable set $S_0\subseteq S$ such that $G\subseteq \bigcup_{s\in S_0}W_s$. Let us observe that $G=(\bigcup_{s\in S_0}W_s)\cap X$.

Denote $T=S$ in case (\ref{perfnorm}) and $T=S_0$ in case (\ref{lind}).   For every  $k\in\mathbb N$ we put
   \begin{gather*}
   T_{k}=\{s\in T: W_{s} \mbox{\,\,depends on the first\,\,} k \mbox{\,\,coordinates}\}, \\
   U_{k}=\bigcup_{s\in T_{k}} W_{s}.
   \end{gather*}
Then $U_k$ depends on the first $k$ coordinates and $U_k\subseteq U_{k+1}$ for every $k\in\mathbb N$. Since for every $k\in\mathbb N$ the set $U_k$ is functionally open in $P$, there exist  sequences $(F_{k,n})_{n=1}^\infty$ and $(G_{k,n})_{n=1}^\infty$  of functionally closed and functionally open subsets of $P$, respectively, such that each of these sets depends on the first $k$ coordinates,
\begin{gather*}
  U_k=\bigcup_{n=1}^\infty F_{k,n}\quad\mbox{and}\quad F_{k,n}\subseteq G_{k,n}\subseteq F_{k,n+1}.
\end{gather*}
For all $n\in\mathbb N$ we set
$$
A_n=\bigcup_{k=1}^n F_{k,n}\quad\mbox{and}\quad B_n=\bigcup_{k=1}^n G_{k,n}.
$$
Then  the set $A_n$ is functionally closed, the set $B_n$ is functionally open, each of them depends on finitely many coordinates and $A_n\subseteq B_n\subseteq A_{n+1}$ for every $n$. Moreover, $G=(\bigcup_{n=1}^\infty A_n)\cap X$. We choose a function $f_n\in {\rm CF}(P,[0,1])$  such that
$$
A_n=f_n^{-1}(1)\quad\mbox{and}\quad P\setminus B_n=f_n^{-1}(0)
$$
for every  $n$. Then $f_{n}(x)\le f_{n+1}(x)$ for all $n\in\mathbb N$ and $x\in P$. Moreover, $\lim_{n\to \infty} f_n(x)=\chi_{G}(x)$ for every  $x\in X$.

{\bf{(b)}.} We put $\mathcal F=(f_n:n\in\mathbb N)$, where $(f_n)_{n=1}^\infty$ is the sequence constructed in (a). Since each function $f_n$ is continuous on $(P,d_{\mathcal F})$, the function $\chi_{G}$ is lower semicontinuous on $(X,d_{\mathcal F})$. Therefore, $G$ is open in $(X,d_{\mathcal F})$.
\end{proof}

\begin{prop}\label{thm:char_amb_stableCF}
Let $P=\prod_{n=1}^\infty X_n$ be a product of completely regular spaces, $X$ be a subspace of $P$ and $A$ be a functionally  $F_\sigma$- and functionally $G_\delta$-set in $X$ simultaneously. If one of the following conditions holds:
\begin{enumerate}
    \item\label{perfnorm} $P$ is perfectly normal, or

    \item\label{lind} $X$ is Lindel\"{o}f,
  \end{enumerate}
  then $\chi_A\in \overline{{\rm CF}(X,[0,1])}^{\,\rm p}$.
\end{prop}

\begin{proof}  Lemma~\ref{cor:char_closed_stableCF} implies the existence of a countable family  $\mathcal F$ of functions from ${\rm CF}(P,[0,1])$ such that $A$ and $X\setminus A$ are $F_\sigma$ in $(X,d_{\mathcal F})$. Applying \cite[Lemma 3]{Bykov} we obtain increasing sequences  $(F_k^i)_{k=1}^\infty$ of closed sets in $(X,d_{\mathcal F})$ for $i=1,2$, satisfying the conditions
\begin{gather*}
A=\bigcup_{k=1}^\infty F_k^1,\quad X\setminus A=\bigcup_{k=1}^\infty F_k^2,\\
d_k=d_{\mathcal F}(F_k^1,F_k^2)>0\quad \forall k\in\mathbb N.
\end{gather*}
For every $k\in\mathbb N$ we define a function $h_k:P\to [0,1]$ by the formula
$$
h_k(x)=\min\{\tfrac{1}{d_k}d_{\mathcal F}(x,F_k^2),1\}.
$$
Now for every $k\in\mathbb N$ we put
$$
f_k=h_k\circ p_{k}
$$
and show that $(f_k)_{k=1}^\infty$ converges to $\chi_A$ pointwisely on  $X$. Fix  $x\in X$. Since
$$
\lim_{k\to\infty}|f_k(x)-\chi_A(x)|=\lim_{k\to\infty}|h_k(p_k(x))-h_k(x)|\le \lim_{k\to\infty} \frac{d_{\mathcal F}(p_{k}(x),x)}{d_k}=0,
$$
we have
$$
\lim_{k\to\infty} f_k(x)=\chi_A(x).
$$
The continuity of $f_k$ on $(X,d_{\mathcal F})$ implies the continuity of $f_k$ on  $X$ with the product topology induced from $P$. Hence, $\chi_A\in\overline{{\rm CF}(X)}^{\,\rm p}$.
\end{proof}

\begin{remark}
  It follows from~\cite[Proposition 2.6]{Karlova:2016:EJMA} that for every map $f\in{\rm H}_1(X,Y)$ the preimage $f^{-1}(V)$ of any open set $V\subseteq Y$ is functionally $F_\sigma$ in $X$, if $X$ is a normal space and $Y$ is a metrizable space.
\end{remark}

\begin{cor}\label{cor:char_amb_stableCF}
Let $P$ and $X$ be as in Proposition~\ref{thm:char_amb_stableCF}. Then ${\rm H}_1^0(X)\subseteq \overline{{\rm CF}(X)}^{\,\rm p}$.
\end{cor}

\begin{proof}
  We consider a function $f\in {\rm H}_1^0(X)$ with a finite image $f(X)=\{y_1,\dots,y_n\}$. Notice that the family $(A_i:i\in\{1,\dots,n\})$ of the sets $A_i=f^{-1}(y_i)$ forms a partition of $X$ by functionally $F_\sigma$-sets. Moreover, $f(x)=\sum_{i=1}^n y_i\cdot \chi_{A_i}(x)$
for every  $x\in X$. Since $y_i\cdot\chi_{A_i}\in \overline{{\rm CF}(X)}^{\,\rm p}$ for every $i$ by Proposition~\ref{thm:char_amb_stableCF}, $f\in  \overline{{\rm CF}(X)}^{\,\rm p}$.
\end{proof}

\section{R-spaces and uniform limits}

For a family $(f_s:s\in S)$ of maps between topological spaces $X$ and $Y$ we define a map $\mathop{\Delta}_{s\in S}f_s:X\to Y^S$ as the following:
$$
\mathop{\Delta}_{s\in S}f_s(x)=(f_s(x))_{s\in S}
$$
for every $x\in X$. We say that $\mathcal F$ is {\it $\Delta$-closed}, if
$$
f\circ (\mathop{\Delta}_{s\in S}f_s)\in\mathcal F
$$
for any finite set $\{f_s:s\in S\}\subseteq\mathcal F$ and a continuous map $f:Y^S\to Y$.

A metric space $(Y,d)$ is said to be an {\it R-space}~\cite{Karlova:2005}, if for every $\varepsilon>0$ there exists a continuous map   $r_{\varepsilon}:Y\times Y\to Y$ with the following properties:
\begin{gather}
  d(y,z)\le\varepsilon\,\, \Longrightarrow\,\, r_\varepsilon(y,z)=y,\\
d(r_\varepsilon(y,z),z)\le\varepsilon
\end{gather}
for all $y,z\in Y$. Let us observe that every convex subset $Y$ of a normed space $(Z,\|\cdot\|)$ equipped with the metric induced from $(Z,\|\cdot\|)$ is an R-space, where the map $r_{\varepsilon}$ is defined as $r_{\varepsilon}(y,z)=z+(\varepsilon/\|y-z\|)\cdot (y-z)$, if $\|y-z\|>\varepsilon$, and $r_{\varepsilon}(y,z)=y$, otherwise.

\begin{thm}\label{thm:unif_lim_Delta}
Let  $X$ be a topological space, $(Y,d)$ be an R-space and $\mathcal F$ be a $\Delta$-closed family of maps between  $X$ and $Y$. Then the family  $\overline{\mathcal F}^{\,{\rm p}}$ is closed under uniform limits.
\end{thm}

\begin{proof} Let $(r_n)_{n=1}^\infty$ be a sequence of continuous functions $r_n:Y\times Y\to Y$ such that
\begin{gather}
  d(y,z)\le \frac{1}{2^n}\,\, \Longrightarrow\,\, r_n(y,z)=y,\\
d(r_n(y,z),z)\le\frac{1}{2^n}
\end{gather}
for all $y,z\in Y$ and $n\in\mathbb N$.

We consider a sequence $(f_n)_{n=1}^\infty$ of maps $f_n\in \overline{\mathcal F}^{\,{\rm p}}$ which is convergent to a map $f:X\to Y$ uniformly on $X$. Without loss of generality we may assume that
$$
d(f_n(x),f_{n+1}(x))\le \frac{1}{2^{n+1}}
$$
for all  $n\in\mathbb N$ and $x\in X$. For every $n$ we take a sequence of maps $g_{n,m}:X\to Y$ such that
\begin{gather*}
\lim_{m\to \infty} g_{n,m}(x)=f_n(x) \quad\forall x\in X,\\
g_{n,m}\in\mathcal F\quad \forall m,n\in\mathbb N.
\end{gather*}
For all $x\in X$ and $m\in {\mathbb N}$ let
\begin{gather*}
\varphi_{1,m}(x)=g_{1,m}(x)\quad\mbox{and}\quad \varphi_{n,m}(x)=r_{n-1}(g_{n,m}(x),\varphi_{n-1,m}(x)) \quad\mbox{for  $n>1$}.
\end{gather*}
Then $\varphi_{n,m}\in \mathcal F$ for all $n,m$, since $\mathcal F$ is $\Delta$-closed. Moreover,  the inequality
$$
d(\varphi_{n+1,m}(x),\varphi_{n,m}(x))\le \frac{1}{2^n}
$$
holds for all $n\in\mathbb N$ and $x\in X$.

We claim that the sequence $(\varphi_{n,m})_{m=1}^{\infty}$ converges to $f_n$ pointwisely on $X$ for every $n$. Indeed, fix  $x\in X$ and show that for every  $n\in\mathbb N$ there exists a number $m_n$ such that $\varphi_{n,m}(x)=g_{n,m}(x)$ for all $m\ge m_n$.  For $n=1$ the statement is evident. Suppose that $n>1$  and  $\varphi_{n-1,m}(x)=g_{n-1,m}(x)$ for $m\ge m_{n-1}$. Since $(g_{n,m}(x))_{m=1}^\infty$ converges to $f_n(x)$ and $(g_{n-1,m}(x))_{m=1}^\infty$ converges to $f_{n-1}(x)$, there exists a number $m_0$ such that
\begin{gather*}
d(g_{n,m}(x),f_n(x))< \frac{1}{2^{n+1}}\quad\mbox{and}\quad d(g_{n-1,m}(x),f_{n-1}(x))<\frac{1}{2^{n+1}}
\end{gather*}
for all $m\ge m_0$. Denote $m_n=\max\{m_0,m_{n-1}\}$ and notice that
\begin{gather*}
d(g_{n,m}(x),\varphi_{n-1,m}(x))=d(g_{n,m}(x),g_{n-1,m}(x))\le\\
\le d(g_{n,m}(x),f_n(x))+d(f_n(x),f_{n-1}(x))+d(f_{n-1}(x),g_{n-1,m}(x))\le\\
\le\frac{1}{2^{n+1}}+\frac{1}{2^{n}}+\frac{1}{2^{n+1}}=\frac{1}{2^{n-1}}
\end{gather*}
for all $m\ge m_n$. Consequently,
$$
\varphi_{n,m}(x)=r_{n-1}(g_{n,m}(x),\varphi_{n-1,m}(x))=g_{n,m}(x)
$$
for all $m\ge m_n$.

Finally, we prove that $\lim_{m\to\infty}\varphi_{m,m}(x)=f(x)$. Fix $\varepsilon>0$ and take  $n_0\in \mathbb N$ such that
\begin{gather*}
\frac{1}{2^{n_0-1}}<\frac{\varepsilon}{3}\quad\mbox{and}\quad d(f_{n_0}(x),f(x))\le\frac{\varepsilon}{3}.
\end{gather*}
Since $\lim_{m\to\infty}{\varphi_{n_0,m}}(x)=f_{n_0}(x)$, there exists $k_0>n_0$ such that
$$
d(\varphi_{n_0,m}(x),f_{n_0}(x))\le\frac{\varepsilon}{3}
$$
 for all $m\ge k_0$. Then
\begin{gather*}
d(\varphi_{m,m}(x),f(x))\le \sum_{i=n_0}^{m-1}
d(\varphi_{i+1,m}(x),\varphi_{i,m}(x))+d(\varphi_{n_0,m}(x),f_{n_0}(x))+d(f_{n_0}(x),f(x))\le\\
\le\sum_{i=n_0}^{m-1}\frac{1}{2^i}+\frac{\varepsilon}{3}+\frac{\varepsilon}{3}<\frac{1}{2^{n_0-1}}+\frac{\varepsilon}{3}+\frac{\varepsilon}{3}<\varepsilon
\end{gather*}
 for all $m\ge k_0$.  Therefore, $f\in \overline{\mathcal F}^{\,\rm p}$.
\end{proof}

We will use the following statement from \cite{Karlova:2004} (see also  \cite[\S 31.VIII, Theorem 3]{Kuratowski:Top:1} for metrizable domains).

\begin{lem}\label{lem:unif_lim}
   Let $X$ be a normal space and $Y$ be a metric separable space. Then every map $f\in {\rm H}_1(X,Y)$ is a limit of uniformly convergent sequence of functions $f_n\in {\rm H}_1(X,Y)$ with countable discrete image $f_n(X)$ for every $n\in\mathbb N$. If, moreover, $Y$ is completely bounded, we may assume that each function $f_n$ takes finitely many values.
\end{lem}

\section{Approximation of maps with path-connected range}

\begin{lem}\label{lem:pointwise_countable}
Let $P$ and $X$ be as in Proposition~\ref{thm:char_amb_stableCF} and let $f\in{\rm H}_1(X,\mathbb R)$ has a discrete countable image $f(X)$. Then $f\in \overline{{\rm CF}(X)}^{\,\rm p}$.
\end{lem}

\begin{proof}
  We consider a homeomorphism $\varphi:\mathbb R\to (0,1)$ and the function $\varphi\circ f:X\to (0,1)$ of the first  Lebesgue class. By Lemma~\ref{lem:unif_lim} we choose a sequence $(g_n)_{n=1}^\infty$ of functions from ${\rm H}_1^0(X,(0,1))$ which is convergent to  $\varphi\circ f$ uniformly on $X$. Since $g_n\in \overline{{\rm CF}(X,(0,1))}^{\,\rm p}$ according to Corollary~\ref{cor:char_amb_stableCF} for every $n\in\mathbb N$ and the space $Y=(0,1)$ with the metric $d(x,y)=|x-y|$ is an  R-space, $\varphi\circ f\in \overline{{\rm CF}(X,(0,1))}^{\,\rm p}$ by Theorem~\ref{thm:unif_lim_Delta}. Hence, $f\in \overline{{\rm CF}(X)}^{\,\rm p}$.
\end{proof}

\begin{prop}\label{prop:finite_valued_in_arc}
Let $P$ and $X$ be as in Proposition~\ref{thm:char_amb_stableCF} and let $Y$ be a path-connected space. Then every Lebesgue one function $f:X\to Y$ with countable discrete image $f(X)$ is a pointwise limit of a sequence of functions from ${\rm CF}(X,Y)$.
\end{prop}

\begin{proof}
 Let $f\in {\rm H}_1(X,Y)$ and $f(X)=\{y_n:n\in\mathbb N\}$, $y_n\ge y_m$ for $n\ge m$. For every $n\in\mathbb N$ we take a continuous map $\gamma_n:[n,n+1]\to Y$ such that
$\gamma_n(n)=y_n$ and $\gamma_n(n+1)=y_{n+1}$. Now we define a continuous map $\gamma:[1,+\infty)\to Y$ such that $\gamma|_{[n,n+1]}=\gamma_n$
for every $n\in\mathbb N$. Notice that every set $A_n=f^{-1}(y_n)$ is functionally $F_\sigma$ in $X$ and the family $(A_n:n\in\mathbb N)$ forms a partition of  $X$. We put
$g(x)=n$ if $x\in A_n$ for some $n\in\mathbb N$. Then $g\in {\rm H}_1(X,[1,+\infty))$ and the space $g(X)$ is discrete and countable. By Lemma~\ref{lem:pointwise_countable} there exists a sequence  of functions $g_m\in {\rm CF}(X)$ which is convergent to $g$ pointwisely on $X$. Let us define functions $f_m\in {\rm CF}(X,Y)$ by the formula  $f_m=\gamma\circ g_m$ for $m\in\mathbb N$. It remains to notice that for every  $x\in X$ there exists $n\in\mathbb N$ such that $x\in A_n$, therefore, $\lim_{m\to\infty} f_m(x)=\gamma(g(x))=\gamma(n)=y_n=f(x)$.
\end{proof}

\begin{thm}\label{thm:L_one_is_CF_limit} Let $P$ and $X$ be as in Proposition~\ref{thm:char_amb_stableCF}  and $(Y,d)$ be a path-connected separable R-space.  Then every Lebesgue one function $f:X\to Y$ is a pointwise limit of a sequence of functions from ${\rm CF}(X,Y)$.
\end{thm}

\begin{proof} Let $f\in{\rm H}_1(X,Y)$. According to Lemma~\ref{lem:unif_lim} there exists a sequence of functions $f_n\in{\rm H}_1(X,Y)$ with discrete  countable image which is convergent to $f$ uniformly on $X$. Then $f_n\in \overline{{\rm CF}(X,Y)}^{\,\rm p}$ for every $n$ by Proposition~\ref{prop:finite_valued_in_arc}. Since the family ${\rm CF}(X,Y)$ is $\Delta$-closed,  $f\in \overline{{\rm CF}(X,Y)}^{\,\rm p}$ by Theorem~\ref{thm:unif_lim_Delta}.
\end{proof}

\section{Semicontinuous functions}

\begin{thm}\label{thm:semicont}
Let $P$ and $X$ be as in Proposition~\ref{thm:char_amb_stableCF} and let $f:X\to [0,+\infty)$ be a lower semicontinuous function. Then $f$ is a pointwise limit of an increasing sequence of functions $f_n\in {\rm CF}(X,[0,+\infty))$.
\end{thm}

\begin{proof} Let $\mathbb Q=\{q_n:n\in\mathbb N\}$ and $B_n=(q_n,+\infty)$. Since $f$ is lower semicontinuous, every set
$G_n=f^{-1}(B_n)$ is open in $X$. Lemma~\ref{cor:char_closed_stableCF} implies the existence of countable family $\mathcal F=(f_n:n\in\mathbb N)$ of functions  $f_n\in{\rm CF}(P,[0,1])$ such that every set $G_n$ is open in  $(X,d_{\mathcal F})$. Moreover, for every $a\ge 0$ there exists a sequence $(n_k)_{k=1}^\infty$ such that  $f^{-1}(a,+\infty)=\bigcup_{k=1}^\infty G_{n_k}$. Then $f^{-1}(a,+\infty)$ is open in $(X,d_{\mathcal F})$. Therefore, $f:(X,d_{\mathcal F})\to [0,+\infty)$ is lower semicontinuous.

For all $n\in\mathbb N$  and $x,y\in P$ we put
$$
d_n(x,y)=\sum_{k=1}^n \frac{1}{2^k}|f_k(x)-f_k(y)|
$$
and observe that
\begin{gather}\label{gath:increase}
  d_n(x,y)\le d_{n+1}(x,y)\le d_{\mathcal F}(x,y)\le d_n(x,y)+\frac{1}{2^n}.
\end{gather}
For all $n\in\mathbb N$ and $x\in X$ let
$$
g_n(x)=\inf_{z\in X} (f(z)+nd_n(x,z)).
$$
Then $g_n$ depends on the first $n$ coordinates, since for every fixed $z\in X$  the function $d_n(x,z)$ depends on on the first $n$ coordinates. Moreover, $g_n(x)\le f(x)$ for every $x\in X$.

Fix $n\in\mathbb N$ and notice that the inequalities
\begin{gather*}
  |g_n(x)-g_n(y)|\le \sup_{z\in X}|nd_n(x,z)-nd_n(y,z)|\le n d_n(x,y)\le nd_{\mathcal F}(x,y)
\end{gather*}
for all $x,y\in X$ imply that $g_n$ is uniformly continuous on $(X,d_{\mathcal F})$.

It follows from~(\ref{gath:increase}) that the sequence $(g_n)_{n=1}^\infty$ increases.

It remains to verify that $\lim_{n\to\infty}g_n(x)=f(x)$ for every $x\in X$. Fix $x\in X$ and $\varepsilon>0$. Take a sequence $(z_n)_{n=1}^\infty$ of points from $X$ with
$$
f(z_n)+nd_n(x,z_n)<g_n(x)+\varepsilon.
$$
Taking into account that $g_n(x)\le f(x)$ for every $n$ and $f(z_n)\ge 0$, we get $\lim_{n\to\infty}d_n(x,z_n)=0$. Then $\lim_{n\to\infty}d_{\mathcal F}(x,z_n)=0$. Since
 $f$ is lower semicontinuous on $(X,d_{\mathcal F})$ at the point $x$, there exists a number $k$ such that $f(z_n)\ge f(x)-\varepsilon$ for all $n\ge k$. Then $g_n(x)>f(z_n)-\varepsilon>f(x)-2\varepsilon$ for all $n\ge k$.
\end{proof}

\begin{cor}
Let $P$ and $X$ be as in Proposition~\ref{thm:char_amb_stableCF} and let $f:X\to\mathbb R$ be a lower semicontinuous function. Then $f$ is a pointwise limit of an increasing sequence of functions from ${\rm CF}(X)$ if and only if $f$ has a minorant from ${\rm CF}(X)$.
\end{cor}

\begin{proof}
  Since the necessity is obvious, we will prove the sufficiency. Assume that $g\in{\rm CF}(X)$ is a minorant for $f$ and put $h=f-g$. Applying Theorem~\ref{thm:semicont} for the function $h:X\to [0,+\infty)$, we obtain the increasing sequence of functions $h_n\in{\rm CF}(X,[0,+\infty))$ which is convergent to  $h$ pointwisely on $X$. It remains to put $f_n=h_n+g$ for every $n\in\mathbb N$.
\end{proof}

\section{Pseudocompact case}\label{sec:LindelofProduct}

We say that a sequence $(f_n)_{n=1}^\infty$ of maps $f_n:X\to Y$ between topological spaces $X$ and $Y$ is  {\it stably convergent to a map $f:X\to Y$ on $X$}, if for every $x\in X$ there exists $k\in\mathbb N$ such that $f_n(x)=f(x)$ for all  $n\ge k$.

\begin{lem}\label{lem:uniform_stable}
 Let $f:X\to Y$ be a map between a topological space $X$ and a metric space $(Y,d_Y)$, $(f_n)_{n=1}^\infty$ be a sequence of maps $f_n:X\to Y$ which is  convergent to $f$ stably on $X$ and let  $(f_{n,k})_{k=1}^\infty$ be a sequence of maps  $f_{n,k}:X\to Y$ which is convergent to $f_n$ uniformly on $X$ for every $n\in\mathbb N$. Then $f\in \overline{\{f_{n,k}:n,k\in\mathbb N\}}^{\,\rm p}$.
\end{lem}

\begin{proof} Without loss of generality we may assume that
 \begin{gather*}
   d_Y(f_{n,k}(x),f_n(x))<\frac{1}{2^k},\\
   d_Y(f_{n,k+1}(x),f_{n,k}(x))<\frac{1}{2^k}
 \end{gather*}
for all $x\in X$ and $n,k\in\mathbb N$.

Fix  $x\in X$ and $\varepsilon>0$. We choose $m\in\mathbb N$ such that
\begin{gather*}
  \frac{1}{2^{m-3}}<\varepsilon\quad\mbox{and}\quad f_n(x)=f(x)\quad\forall n\ge m.
\end{gather*}
Then we get
\begin{gather*}
  d_Y(f_{n,n}(x),f(x))\le \sum\limits_{k=m}^{n-1} d_Y(f_{n,k+1}(x),f_{n,k}(x))+d_Y(f_{n,m}(x),f_n(x))<\\
  <\sum\limits_{k=m}^{n-1}\frac{1}{2^k}+\frac{1}{2^m}<\frac{1}{2^{m-2}}+\frac{1}{2^m}<\frac{1}{2^{m-3}}<\varepsilon
\end{gather*}
 for all $n\ge m$. Therefore, $\lim\limits_{n\to\infty}f_{n,n}(x)=f(x)$.
\end{proof}

\begin{lem}\label{lem:pseudocomp}
  Let $P=\prod_{n=1}^\infty X_n$ be a pseudocompact space and $(Y,d_Y)$ be a path-connected metric space. Then every Baire one function $f:X\to Y$ with countable discrete image $f(X)$ is a pointwise limit of a sequence of functions from ${\rm CF}(X,Y)$.

\end{lem}

\begin{proof}
  It  follows from \cite[Lemma 4]{Karlova:2005} that there exists a sequence $(f_n)_{n=1}^\infty$ of continuous functions $f_n:P\to Y$ which is stably convergent to $f$  on $P$. By \cite[Theorem 1]{Bykov} every function $f_n$ is the uniform limit of a sequence from ${\rm CF}(P,Y)$. It remains to apply Lemma~\ref{lem:uniform_stable}.
 \end{proof}

 \begin{thm}\label{thm:pseudocomp}
  Let $P=\prod_{n=1}^\infty X_n$ be a pseudocompact space and $(Y,d_Y)$ be a path-connected separable metric R-space. Then every Baire one function $f:P\to Y$ is a pointwise limit of a sequence of functions from ${\rm CF}(P,Y)$.
 \end{thm}

   \begin{proof} Let $f\in{\rm B}_1(P,Y)$. It follows from \cite[Lemma 3.3]{Karlova:2004} and \cite[Lemma 4]{Karlova:2005} that there exists a sequence of functions $f_n\in{\rm B}_1(P,Y)$ with discrete countable image $f(P)$ which is uniformly convergent to $f$. Then $f_n\in \overline{{\rm CF}(P,Y)}^{\,\rm p}$ for every $n$ by Lemma~\ref{lem:pseudocomp}. Since the family ${\rm CF}(P,Y)$ is $\Delta$-closed,  $f\in \overline{{\rm CF}(P,Y)}^{\,\rm p}$ by Theorem~\ref{thm:unif_lim_Delta}.
\end{proof}

The authors do not know an answer to the following question.

\begin{quest}
  Let $X\subseteq\prod_{n=1}^\infty X_n$ be a pseudocompact subspace of a product of completely regular spaces $X_n$ and $f:X\to\mathbb R$ be a Baire one function. Does there exist a sequence of functions from ${\rm CF}(X)$ which is pointwisely convergent to $f$ on $X$?
\end{quest}

\section{Example}

Denote by $\bf C$ the Cantor set ${\bf C}\subseteq [0,1]$. For every $n\in\mathbb N$ and $k\in \{1, \dots, n\}$ we denote by $p_{k,n}$ the natural projection $p_{k,n}:{\bf C}^n\to {\bf C}$, $p_{k.n}(x_1, \dots x_n)=x_k$. We put
$$
{\rm r}_n(A)=\bigcup_{k=1}^n p_{k,n}^{-1}(A)
$$
for every $n\in \mathbb N$ and for an infinite set $A\subseteq {\bf C}$.
Following to \cite[Fact 3, p.~124]{Tkachuk}, for every $n\in\mathbb N$ and $P\subseteq {\bf C}^n$ we put
$$
|P|_n=\min\{|A|: A\subseteq {\bf C}\,\,{\rm and}\,\,P\subseteq {\rm r}_n(A)\}.
$$
According to \cite[Fact 4, p.~125]{Tkachuk}, we have that if $|F|_n>\aleph_0$ for some $n\in\mathbb N$ and a closed in ${\bf C}^n$ set $F$ then $|F|_n=\frak{c}$.

\begin{lem}\label{lem:set-sequence}
  There exists a sequence $(A_n)_{n=1}^\infty$ of sets $A_n\subseteq {\bf C}$ such that the following conditions hold:
\begin{enumerate}
 \item[(1)] ${\bf C}=\bigsqcup_{n=1}^\infty A_n$;

\item[(2)] $(A_n)^m\cap F\ne \emptyset$ for all $n,m\in\mathbb N$ and a closed set $F$ in ${\bf C}^m$ with $|F|_m=\frak c$.
\end{enumerate}
\end{lem}
\begin{proof}
For every $m\in\mathbb N$ we denote by ${\mathcal F}_m$ the system of all closed sets $F\subseteq {\bf C}^m$ with $|F|_m=\frak c$. Clearly, $|{\mathcal F}_m|=\frak c$ for every $m\in\mathbb N$. Let $\bigcup_{m=1}^\infty{\mathcal F}_m= \{F_{\xi}:1\leq \xi<\frak c\}$. For every $\xi<\frak c$ we denote by $m_\xi$ an integer $m_\xi\in\mathbb N$ such that $F_\xi\subseteq {\bf C}^{m_\xi}$. By transfinite induction on $\xi$ we construct a family $((A_{\xi,n})_{n=1}^\infty:1\leq \xi<\frak c)$ of sequences $(A_{\xi,n})_{n=1}^\infty$ of finite sets $A_{\xi,n}\subseteq {\bf C}$ which satisfies the following conditions:
\begin{enumerate}
\item[(a)] $A_{\xi,n}\cap A_{\theta,m}=\emptyset$ for $(\xi,n)\ne (\theta,m)$;

\item[(b)] $(A_{\xi,n})^{m_\xi}\cap F_{\xi}\ne\emptyset$ for every $\xi<\frak c$ and $n\in\mathbb N$.
\end{enumerate}

Assume that $1\leq \alpha < \frak c$ and the family $((A_{\xi,n})_{n=1}^\infty:1\leq \xi<\alpha)$ is constructed. We put
$$B_{\alpha}=\bigcup_{1\leq\xi<\alpha}\bigcup_{n=1}^\infty A_{\xi,n}.$$
Clearly, $|B_\alpha|<\frak c$. Since ${\rm r}_{m_\alpha}(F_\alpha)=\frak c$, we can  construct easily by the induction on $n$ a sequence $(A_{\alpha,n})_{n=1}^\infty$  of disjoint finite sets $A_{\alpha,n}\subseteq {\bf C}\setminus B_\alpha$ such that $(A_{\alpha,n})^{m_\alpha}\cap F_{\alpha}\ne\emptyset$ for every $n\in\mathbb N$. It remains to put $$A_n=\bigcup_{1\leq \xi<\alpha} A_{\xi,n}$$
for every $n\geq 2$ and $$A_1={\bf C}\setminus \left(\bigcup_{n\geq 2}A_n\right).$$
\end{proof}

\begin{lem}\label{lem:function-second}
  There exists a function $g:{\bf C}\to\mathbb R$ of the second Baire class such that for every countable set $A\subseteq {\bf C}$ the restriction $g|_{{\bf C}\setminus A}$ is not a function of the first Baire class.
\end{lem}

\begin{proof} Let $\mu$ be the Lebesgue measure on $[0,1]$. We prove the statement for a function $g:X\to\mathbb R$ defined on a Cantor type closed set $X\subseteq [0,1]$ which is homeomorphic to ${\bf C}$ and $\mu(G)>0$ for every open set $G\subseteq X$.

We choose an $F_\sigma$-set $B$ of the first category in $X$ such that $\mu(B)=\mu(X)$ and put $g=\chi|_B$. Since $B$ is $F_\sigma$, the function $g$ belongs to  the second Baire class. Moreover, for every countable set $A\subseteq X$ the function $g|_{X\setminus A}$ is everywhere discontinuous. Therefore, the restriction $g|_{X\setminus A}$ is not a Baire one function, because the space $X\setminus A$ is Baire.
\end{proof}

\begin{thm}\label{th:example}
There exist a sequence $(X_n)_{n=1}^\infty$ of Lindel$\ddot{o}$f spaces $X_n$  and a function $f:\prod_{n=1}^\infty X_n\to\mathbb R$ of the first Baire class such that
\begin{enumerate}
\item[(i)] for every $n\in\mathbb N$ the product $Y_n=\prod_{k=1}^n X_k$ is Lindel$\ddot{o}$f;

\item[(ii)] $f$ is not a pointwise limit of any sequence $(f_n)_{n=1}^\infty$ of functions from ${\rm CF}(\prod_{n=1}^\infty X_n)$.
\end{enumerate}
\end{thm}

\begin{proof} Let $(A_n)_{n=1}^\infty$ be a sequence of sets $A_n\subseteq {\bf C}$ which satisfy conditions $(1)$ and $(2)$ of Lemma \ref{lem:set-sequence}. Let $\tau_0$ be the usual topology of the Cantor set $\bf C$. For every $n\in\mathbb N$ we denote by $\tau_n$ the topology on ${\bf C}$ with the base
$$
\tau_0\cup\{\{x\}:x\in A_n\}.
$$
Put $X_n=({\bf C}, \tau_n)$ for every $n\in\mathbb N$.

{\bf (i)}. Fix $n\in\mathbb N$ and put $B_n=\bigcup_{k=1}^nA_k$. Consider the space $Z=({\bf C},\sigma)$, where $\sigma$ is a topology on ${\bf C}$ with the base
$$
\tau_0\cup\{\{x\}:x\in B_n\}.
$$
According to \cite[Fact 6, p.~126]{Tkachuk}, the space $Z^n$ is Lindel\"{o}f. Since the topology of the product $Y_n$ is weaker than the topology of $Z^n$, $Y_n$ is Lindel\"{o}f too.

{\bf (ii)}. We start with auxiliary properties of functions on the space $X=\prod_{n=1}^\infty X_n$.

{\tt Claim A}. {\it For every function $v:{\bf C}\to\mathbb R$ of the first Baire class there exists a continuous function $u:X\to\mathbb R$ such that $u(a,a,\dots)=v(a)$ for every $a\in {\bf C}$.}

{\tt Proof of Claim A.} Choose a sequence $(v_n)_{n=1}^\infty$ of continuous functions $v_n:{\bf C}\to\mathbb R$ which converges to $v$ pointwisely on ${\bf C}$. Using the uniform continuity of the functions $v_n$, we choose a strictly decreasing sequence of reals $\delta_n>0$ such that $|v_n(x')-v_n(x'')|<\frac{1}{n}$ for every $n\in\mathbb N$ and $x',x''\in {\bf C}$ with $|x'-x''|<\delta_n$. For every $n\in\mathbb N$ we put
$$
G_n=\{(x_k)_{k=1}^{\infty}\in X: \max\{|x_i-x_j|:1\leq i,j\leq n\}<\delta_n\}
$$
and
$$
F_n=\{(x_k)_{k=1}^{\infty}\in X: \max\{|x_i-x_j|:1\leq i,j\leq n\}\leq\delta_n\}.
$$
It is evident  that every set $G_n$ is open and every set $F_n$ is closed  in ${\bf C}^\omega$. We put $v_0(x)=0$ for every $x\in {\bf C}$ and
$$
\Delta=\{(x,x,\dots):x\in{\bf C}\}.
$$
For every $n\in\mathbb N$ we choose a continuous function $\varphi_n:{\bf C}^\omega\to[0,1]$ such that $\varphi_n(X\setminus G_{n})\subseteq \{1\}$ and $\varphi_n(F_{n+1})\subseteq \{0\}$ and consider the function
$u:X\to \mathbb R$,
$$
 u(x)=\left\{\begin{array}{lll}
                         v_0(x), & x\in X\setminus F_1\\
                         \varphi_n(x)v_{n-1}(x_1)+(1-\varphi_n(x)v_n(x_1), & x=(x_1,x_2,\dots)\in F_{n}\setminus F_{n+1}\\
                         v(a), & x=(a,a,\dots)\in \Delta.
                       \end{array}
 \right.
 $$
 It easy to see that the function $u$ is continuous at every point $x\in X\setminus \Delta$. It remains to verify the continuity of $u$ at every point $x\in \Delta$. Fix $a\in{\bf C}$ and $\varepsilon>0$. We choose $m\in\mathbb N$ such that $a\in A_m$ and let $n_0\geq\max\{m,\frac{2}{\varepsilon}+1\}$ be such that $|v_n(a)-v(a)|<\frac{\varepsilon}{2}$ for all $n\geq n_0-1$. Now we consider the neighborhood
 $$
 U=\{(x_n)_{n=1}^{\infty}\in X:x_{m}=a\,\,{\rm and}\,\,\max\{|x_i-x_j|:1\leq i,j\leq n_0\}\leq\delta_{n_0}\}
 $$
of the point $x_0=(a,a,\dots)$ in  $X$. Fix $x=(x_n)_{n=1}^\infty\in U\setminus \{x_0\}$ and take $k\in\mathbb N$ such that $x\in F_{k}\setminus F_{k+1}$. Notice that $k\geq n_0$, because $U\subseteq F_{n_0}$. Therefore, $$|v_{k}(a)-v(a)|<\frac{\varepsilon}{2}\,\,\,\,\,\,{\rm and}\,\,\,\,\,\, |v_{k-1}(a)-v(a)|<\frac{\varepsilon}{2}.$$ Since $x_m=a$, $|x_1-a|<\delta_k$ and, moreover $|x_1-a|<\delta_{k-1}$. Therefore, $$|v_k(x_1)-v_k(a)|<\tfrac{1}{k}\leq \tfrac{1}{n_0}\leq \tfrac{\varepsilon}{2}$$
and
$$|v_{k-1}(x_1)-v_{k-1}(a)|<\tfrac{1}{k-1}\leq \tfrac{1}{n_0-1}\leq \tfrac{\varepsilon}{2}.$$
Thus,
$$|v_{k}(x_1)-v(a)|<\varepsilon\,\,\,\,\,\,{\rm and}\,\,\,\,\,\, |v_{k-1}(a)-v(a)|<\varepsilon.$$
Now we have
$$
|u(x)-v(a)|=|\varphi_k(x)(v_{k-1}(x_1)-v(a))+(1-\varphi_k(x)(v_k(x_1)-v(a)|\leq
$$
$$
\varphi_k(x)|v_{k-1}(x_1)-v(a)|+(1-\varphi_k(x)|v_k(x_1)-v(a)|<\varphi_k(x)\varepsilon +(1-\varphi_k(x)\varepsilon=\varepsilon.
$$
{\tt Claim B}. {\it For every function $w:{\bf C}\to\mathbb R$ of the second Baire class the function $h:X\to\mathbb R$
$$
 h(x)=\left\{\begin{array}{ll}
                         0, & x\in X\setminus \Delta\\
                         w(a), & x=(a,a,a,\dots)\in \Delta
                       \end{array}
 \right.
 $$
 is a function of the first Baire class.}

{\tt Proof of Claim B.} We choose a sequence $(w_n)_{n=1}^\infty$ of Baire one functions $w_n:{\bf C}\to\mathbb R$ which converges to $w$ pointwisely on ${\bf C}$. By Claim A we can construct a sequence $(h_n)_{n=1}^\infty$ of continuous functions $h_n:X\to\mathbb R$ such that $h_n(a,a,\dots)=w_n(a)$ for all $a\in X$ and $h_n(x)=0$ for all $n\in\mathbb N$ and $x\in X\setminus C_n$, where
$$
C_n=\{(x_k)_{k=1}^{\infty}\in X: \max\{|x_i-x_j|:1\leq i,j\leq n\}\leq\tfrac1n\}.
$$
Clearly, the sequence $(h_n)_{n=1}^\infty$ converges to $h$ pointwisely on $X$.

Now we prove the condition $(ii)$.  Lemma \ref{lem:function-second} implies the existence of a function $g:{\bf C}\to\mathbb R$ of the second Baire class such that for every countable set $A\subseteq {\bf C}$ the restriction $g|_{{\bf C}\setminus A}$ does not belong to the first Baire class. According to Claim B, the function $f:X\to\mathbb R$
$$
 f(x)=\left\{\begin{array}{ll}
                         0, & x\in X\setminus \Delta\\
                         g(a), & x=(a,a,a,\dots)\in \Delta
                       \end{array}
 \right.
 $$
 is a function of the first Baire class.

Now we show that (ii) holds. Suppose to the contrary that $f$ is the pointwise limit of a sequence   of  functions $f_n\in{\rm CF}(X)$. Without loss of generality we may assume that for every $n\in\mathbb N$ the function $f_n$ depends on the first $n$ coordinates. For every $n\in\mathbb N$ and $z\in{\bf C}$ we put
$$
g_n(z)=f_n(z,z,\dots).
$$
Fix $n\in\mathbb N$. Since  $f_n$ depends on the first $n$ coordinates, the function $g_n:{\bf C}\to\mathbb R$ is continuous at every point $z\in{\bf C}\setminus(\bigcup_{k=1}^n A_k)$. It follows from condition (2) of Lemma \ref{lem:set-sequence} that every closed set $F\subseteq \bigcup_{k=1}^n A_k$ is at most countable. Therefore, an $F_\sigma$-set $D_n$ of all discontinuity points of $f_n$ is at most countable too.

Hence, the set $D=\bigcup_{n=1}^\infty D_n$ is at most countable, every function $g_n$ is continuous at each point $z\in {\bf C}\setminus D$ and the sequence $(g_n)_{n=1}^\infty$ converges to $g$ pointwisely on ${\bf C}$, which contradicts to the choice of $g$.
\end{proof}

\begin{remark}\label{r.1}
$1)$ It easy to see that for a product $X=\prod_{n=1}^\infty X_n$ of a sequence $(X_n)_{n=1}^\infty$ of topological spaces $X_n$ and for a continuous map $f:X\to Y$ with values in a topological space $Y$ there exists a sequence $(f_n)_{n=1}^\infty$ of  functions $f_n\in{\rm CF}(X,Y)$ which converges to $f$ pointwisely on $X$.

$2)$ In fact it was proved in Theorem~\ref{thm:L_one_is_CF_limit} that for appropriate spaces $X\subseteq \prod_{n=1}^\infty X_n$ and $Y$ for every mapping $f:X\to Y$ of the first Baire class there exists a sequence  of  functions from ${\rm CF}(\prod_{n=1}^\infty X_n,Y)$ such that the sequence of the restrictions $f_n|_{X}$ converges to $f$ pointwisely on $X$.

$3)$ For the sequence $(X_n)_{n=1}^\infty$ of topological spaces $X_n$ from Theorem \ref{th:example}, for the subspace $X=\Delta$ and  $f\in {\rm CF}(X)$ with $f(x,x,\dots)=g(x)$, there does not exist a sequence   of  functions $f_n:\prod_{n=1}^\infty X_n\to \mathbb R$ such that the sequence of the restrictions $f_n|_{X}$ converges to $f$ pointwisely.
\end{remark}

\end{document}